\documentclass[11pt,a4paper]{article}
\usepackage{amsmath}
\usepackage{amsfonts}
\usepackage{amssymb}
\usepackage{centernot}
\usepackage[margin=3cm]{geometry}
\usepackage[colorlinks=true,linkcolor=blue,citecolor=blue]{hyperref}
\usepackage{enumerate}
\usepackage{graphicx}
\usepackage{algorithm2e}
\usepackage{tikz}
\usetikzlibrary{calc}

\usepackage[colorinlistoftodos]{todonotes}

\usepackage{amsthm}
\newtheorem{theorem}{Theorem}[section]
\newtheorem{proposition}{Proposition}[section]
\newtheorem{corollary}{Corollary}[section]

\newtheorem{lemma}{Lemma}[section]
\theoremstyle{remark}
\newtheorem{remark}{Remark}[section]
\newtheorem{example}{Example}[section]
\newtheorem{assumption}{Assumption}[section]

\newcommand{\qede}{\hspace*{\fill}$\Diamond$\medskip}
\newcommand{\Hilbert}{\mathcal{H}}
\newcommand{\R}{\mathbb{R}}
\newcommand{\N}{\mathbb{N}}

\def\DR{\mathop{\rm DR}\nolimits}
\def\tto{\rightrightarrows}

\title{Global Behavior of the Douglas--Rachford Method\\
       for a Nonconvex Feasibility Problem}

\author{Francisco J. Arag\'on Artacho\thanks{Department of Statistics and Operations Research,
University of Alicante, \textsc{Spain}. e-mail:~\url{francisco.aragon@ua.es}}
        \and Jonathan M. Borwein\thanks{Centre for Computer-Assisted Research Mathematics and its Applications (CARMA), University of Newcastle, Callaghan, NSW 2308, \textsc{Australia}. e-mail:~\url{jon.borwein@gmail.com}}
        \and Matthew K. Tam\thanks{Centre for Computer-Assisted Research Mathematics and its Applications (CARMA), University of Newcastle, Callaghan, NSW 2308, \textsc{Australia}. e-mail:~\url{matthew.tam@uon.edu.au}}}

\begin{document}
\maketitle

\begin{abstract}
 In recent times the Douglas--Rachford algorithm has been observed empirically to solve a variety of nonconvex feasibility problems including those of a combinatorial nature. For many of these problems current theory is not sufficient to explain this observed success and is mainly concerned with questions of local convergence. In this paper we analyze global behavior of the method for finding a point in the intersection of a half-space and a potentially non-convex set which is assumed to satisfy a well-quasi-ordering property or a property weaker than compactness. In particular, the special case in which the second set is finite is covered by our framework and provides a prototypical setting for combinatorial optimization problems.
 \end{abstract}

\paragraph*{Keywords} Douglas--Rachford algorithm, global convergence, feasibility problem, half-space, non-convex
\paragraph*{MSC2010:} 90C26, 65K05

\section{Introduction}
 Recent computational experiments have demonstrated the ability of \emph{Douglas--Rachford methods} to successfully solve a variety of significant non-convex optimization problems. Examples include combinatorial optimization \cite{ABTcomb,ABTmatrix,ERTsearching}, low-rank matrix reconstruction \cite{BTreflection}, boolean satisfiability \cite{GEdivide}, sphere packing \cite{GEdivide,ERTsearching}, matrix completion \cite{ABTmatrix}, image reconstruction \cite{BCLphase}, and road design \cite{BKroad}. Intriguingly, success in such settings is not uniform \cite{ABTcomb}, and thus the theory is in need of significant enhancement.

 Douglas--Rachford methods belong to the family of so called \emph{projection algorithms}  and are traditionally analyzed using nonexpansivity properties when the problem is convex. For non-convex problems the theory is still in its infancy. One approach to proving convergence is to replace assumptions of convexity with regularity properties not tantamount to convexity \cite{HLnonconvex,Plinear,BNlocal,HNLsparse}. Examples of the properties required include: \emph{superregularity}, \emph{strong regularity of sets}, \emph{$(\epsilon,\delta)$-regularity}, and exploiting convex structure ({\em i.e.,} realizing a non-convex constraint as the union of certain convex sets). Being local in nature, these assumptions lead to local results; that is, they hold in a sufficiently small neighborhood of a solution. An alternate, albeit less general, approach for proving convergence is to instead focus on specific non-convex feasibility problems. In this direction, global convergence of the method has been established for the two set feasibility problem involving a line (or hyperplane) and sphere \cite{BSabense,ABglobal,benoist}.

 In the non-convex feasibility problems for which Douglas--Rachford methods have been successful, it is the methods' apparent global convergence properties which deserve greater attention, and hence on which we shall focus. This is the case, for instance, in problems having discrete or combinatorial constraints where local convergence can be deduced from the general theory of convex sets as a consequence of the \emph{local convexity} of the constraints. Furthermore, an implementation of a Douglas--Rachford method would always \emph{round} the current iterate, and check if this rounded iterate is a solution to the problem. That is, in practice, the algorithm is never \emph{run locally}.

  The purpose of this paper is to analyze global properties of the basic Douglas--Rachford method. Of the two approaches already discussed, ours bears more resemblance to the latter. We focus on a two set feasibility problem with the first set allowed to be quite general --- it must satisfy either an assumption which encompasses all (weakly) compact sets, or a well-quasi-ordering property with respect to a quasi-order induced by the distance function. Both assumptions always hold for finite sets.  The trade-off for this assumption is that a more constrained structure is required of the second set ---  it must be a closed half-space. The fact that this analysis is far from immediate, illustrates the subtleness of the method's global behavior, even for common problems.

  The remainder of this paper is organized as follows. In Section~\ref{sec:prelim} we introduce the problem to be solved and the basic Douglas--Rachford algorithm. In Section~\ref{sec:properties} we study properties of the Douglas--Rachford operator. In Section~\ref{sec:analysis} we analyze the global behavior of the method and give our main results. In Section~\ref{sec:examples} we give various examples and counter-examples. Finally, in Section~\ref{sec:conclusion}, we make our concluding remarks.

\section{Preliminaries}\label{sec:prelim}
Through this paper our setting is the real Hilbert space $\Hilbert$ equipped with inner-product $\langle\cdot,\cdot\rangle$ and induced norm $\|\cdot\|$. We consider the \emph{feasibility problem}
 \begin{equation}
  \text{Find }x\in H\cap Q,\label{eq:Problem}
 \end{equation}
where $Q\subset\Hilbert$ is a closed set, and $H\subset\Hilbert$ is a (closed) half-space. We will be concerned with the case in which $Q$ has additional properties (see Assumptions~\ref{ass:qwo} \&~\ref{ass:compactQH}), but is intended to be as general as possible -- these assumptions are explicitly stated where needed. It is convenient to represent $H$, and its \emph{dividing hyperplane} $L$, in the form
$$H:=\{x\in\Hilbert\mid \langle a,x\rangle\leq b\},\qquad L:=\{x\in\Hilbert\mid \langle a,x\rangle =b\},$$
where $b\in\R$, and $a\in\Hilbert$ with $\|a\|=1$.

To solve \eqref{eq:Problem} we employ the specific \emph{Douglas--Rachford operator}, denoted $\DR:\Hilbert\tto\Hilbert$,  given by
 \begin{equation}
  \DR(x):=\frac{x+R_H(R_Q(x))}{2}.\label{eq:DR}
 \end{equation}
Here $R_A$ denotes the \emph{reflector} of a point $x$ with respect to the set $A\subset\Hilbert$  given by $R_A(x):=2P_A(x)-x$, and $P_A$ stands for the \emph{projector} of $x$ onto $A$ given by
 $$P_A(x):=\left\{z\in A\;\middle|\; \|x-z\|=\inf_{a\in A}\|x-a\|\right\}.$$
When $x$ is a \emph{fixed point} of the Douglas--Rachford operator ({\em i.e.,} $x\in\DR(x)$) there is an element of $P_Q(x)$ which solves \eqref{eq:Problem} as is easy to confirm. This suggests that iterating the Douglas--Rachford operator to find a fixed point is a potential method for solving \eqref{eq:Problem}.

Given an initial point $x_0\in\Hilbert$, we say the sequence $\{x_k\}_{k\in\mathbb{N}}$ is a \emph{Douglas--Rachford iteration} if
\begin{equation*}
x_{k+1}\in\DR(x_k)\quad\text{for}\quad k\in\mathbb{N}.
\end{equation*}

We now make explicit the precise form of the Douglas--Rachford operator for our feasibility problem in \eqref{eq:Problem}. The projections onto $L$ and $H$ are given by
$$P_L(x)=x-(\langle a,x\rangle-b)a,\quad
P_H(x)=\begin{cases}
        x           & \text{if } \langle a,x\rangle\leq b,\\
        x-(\langle a,x\rangle-b)a &\text{if }\langle a,x\rangle>b.
       \end{cases}$$
 We assume $P_Q$ is  effectively computable. The corresponding reflections are given by
\begin{equation*}
R_L(x)=x-2(\langle a,x\rangle-b)a,\quad
R_H(x)=\begin{cases}
x & \text{if } \langle a,x\rangle\leq b,\\
x-2(\langle a,x\rangle-b)a &\text{if }\langle a,x\rangle>b,
\end{cases}
\end{equation*}
and $R_Q(x)=2P_Q(x)-x$.
The Douglas--Rachford operator~\eqref{eq:DR} may therefore be expressed as
  $$\DR(x)=\bigcup_{q\in P_Q(x)}\DR(x,q),$$
where
\begin{equation}\label{eq:DRiteration}
 \DR(x,q):=\begin{cases}
            q &\text{if }\langle a,2q-x\rangle\leq b,\\
            q+(\langle a,x\rangle+b-2\langle a,q\rangle)a &\text{if }\langle a,2q-x\rangle> b.
          \end{cases}
\end{equation}
Under this notation, a sequence $\{x_k\}_{k\in\mathbb{N}}$ is a Douglas--Rachford iteration if
 $$x_{k+1}=\DR(x_k,q_k)\quad\text{for some}\quad q_k\in P_Q(x_k).$$
We shall refer to the sequence $\{q_k\}_{k\in\mathbb{N}}$ as an \emph{auxiliary sequence} for $\{x_k\}$.

An example of a Douglas--Rachford iteration in the plane when $Q$ is a  set of  four points is given in Figure~\ref{Fig:1}.  The iteration converges to a solution of \eqref{eq:Problem} in eight steps. This behavior is explained by Theorem~\ref{th:finiteQ}.

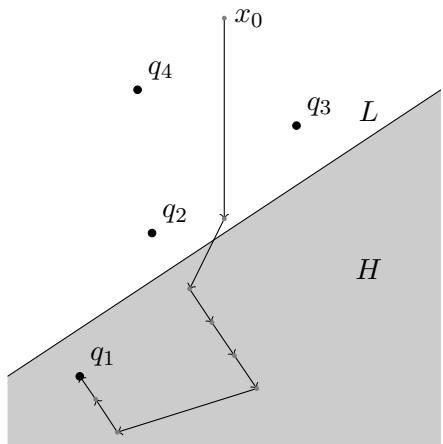
\begin{figure}[ht!]
\begin{center}
    \begin{tikzpicture}[scale=0.95]
	 \newcommand{\SetQ}{(-2,-2),(-1,0),(1,1.5),(-1.2,2)};
	 \fill[black!20] (-3,-2) -- (-3,-3) -- (3,-3) -- (3,2);
	 \draw[black] (-3,-2)  -- (3,2);
	 \draw[black] (2,-0.5) node {$H$};
	 \draw[black] (2,1.7) node {$L$};
	 \coordinate (x0) at (0,3) ;
	 \draw (x0) node[right] {$x_0$};
	 \newcommand{\DRseq}{(0,0.2),(-0.48,-0.78),(-0.17,-1.25),(0.14,-1.71),(0.45,-2.17),(-1.48,-2.78),(-1.78,-2.32),(-2,-2)};
	 \foreach \x [count=\j] in \DRseq {
	   \pgfmathsetmacro{\jj}{\j-1}
	   \draw \x coordinate (x\j);
	   \draw[black,->] (x\jj) -- (x\j);
	   \draw (x\jj) node[draw=black!50,circle,inner sep=0.5pt,fill=black!50,anchor=center] {};
	 };
	 \draw[green] (-2,-2) node[draw,circle,inner sep=0.5pt,fill=black,anchor=center] {};
	 \foreach \q [count=\j] in \SetQ {
	   \draw \q node[draw,circle,inner sep=1pt,fill=black] {};
	   \draw \q node[above right] (q\j) {$q_{\j}$};
	 };
	\end{tikzpicture}
    \caption{A Douglas--Rachford iteration in $\R^2$ with $Q=\{q_1,q_2,q_3,q_4\}$.}\label{Fig:1}
\end{center}
\end{figure}

\begin{remark}
 In practice, an implementation of the Douglas--Rachford iteration would hope to terminate as soon as $q_{k_0}\in Q\cap H$ for some $k_0\in\N$. Observe, however, that~\eqref{eq:DRiteration} does not necessarily ensure that the Douglas--Rachford sequence or its auxiliary sequence remain constant for $k\geq k_0$. It is therefore important to distinguish the Douglas--Rachford iteration from an algorithm arising from its implementation (see Algorithm~\ref{alg:qwo}). \qede
\end{remark}

 It is worth emphasizing that, unlike the Douglas--Rachford algorithm, other projection algorithms can fail when applied to \eqref{eq:Problem}. An example is given in Example~\ref{ex:APfailure}. The simplest method from this family is the \emph{alternating projection algorithm}. It iterates by alternatively applying projectors onto $Q$ and $H$. Precisely, given an initial point $x_0\in\Hilbert$, it generates a sequence $\{x_k\}$ given by $x_{k+1}\in P_H(P_Q(x_k))$.
\begin{example}[Failure of alternating projections]\label{ex:APfailure}
 In general, von Neumann's alternating projection is unable to find a point in the intersection of $H$ and $Q$ (and hence the same is true for the \emph{cyclic Douglas--Rachford algorithm} \cite{cycDR1,cycDR2}). Figure~\ref{Fig:2} shows a simple example with a doubleton $Q=\{q_1,q_2\}\subset\R^2$. In this example, $P_Q(P_H(q_1))=q_1$ and the algorithm cycles between $q_1$ and $P_H(q_1)$ for any starting point $x_0\in P_Q^{-1}(q_1)$. \qede

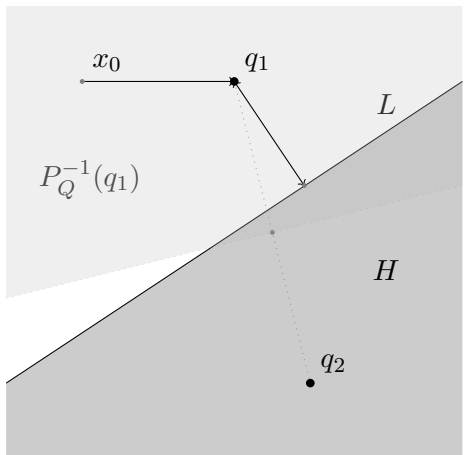
\begin{figure}[ht!]
    \begin{center}
	\begin{tikzpicture}
	 \fill[black!20] (-3,-2) -- (-3,-3) -- (3,-3) -- (3,2);
	 \draw[black] (-3,-2)  -- (3,2);
	 \draw[black] (2,-0.5) node {$H$};
	 \draw[black] (2,1.7) node {$L$};
	 \draw (0,2) coordinate (q1);
	 \draw (1,-2) coordinate (q2);
	 \draw[black!40,dotted] (q1) -- (q2);
	 \draw (0.5,0) node[draw=black!50,circle,inner sep=0.5pt,fill=black!50,anchor=center] {};
	 \fill[black!25,opacity=0.25] (-3,-0.875) -- (-3,3) -- (3,3) -- (3,0.625);
	 \draw[black!50,opacity=0.25,dotted] (-3,-0.875) -- (3,0.625);
	 \draw[black!70] (-1.9,0.7) node {$P^{-1}_Q(q_1)$};
	 \draw (0.92,0.62) coordinate (x2);
	 \draw (-2,2) node[draw=black!50,circle,inner sep=0.5pt,fill=black!50,anchor=center] (x0) {};
	 \draw (x0) node[above right] {$x_0$};
	 \draw (0,2) coordinate (x1);
	 \draw[black,->] (x0) -- (x1);
	 \draw (x1) node[draw=black!50,circle,inner sep=0.5pt,fill=black!50,anchor=center] {};
	 \draw (0.92,0.62) coordinate (x2);
	 \draw[black,<->] (x2) -- (x1);
	 \draw (x2) node[draw=black!50,circle,inner sep=0.5pt,fill=black!50,anchor=center] {};
	 \draw (q1) node[draw,circle,inner sep=1pt,fill=black] {};
	 \draw (q1) node[above right] {$q_1$};
	 \draw (q2) node[draw,circle,inner sep=1pt,fill=black] {};
	 \draw (q2) node[above right] {$q_2$};
	\end{tikzpicture}
    \end{center}
    \caption{Failure of the alternating projection algorithm for initial points in  $P^{-1}_Q(q_1)$.}\label{Fig:2}
\end{figure}
\end{example}

 We next give some examples of feasibility problems of the form in \eqref{eq:Problem} to which our results apply.

\begin{example}[Finite union of compact convex sets]
 Suppose $Q$ is a finite union of compact convex sets in $\R^m$. Applied to this problem, \cite[Theorem~1]{BNlocal} yields a local result that applies near \emph{strong fixed-points} of $\DR(\cdot)$ (these are points for which $\DR(x)=\{x\}$). Namely, that the Douglas--Rachford sequence locally converges to a \emph{fixed-point} and any cluster of its auxiliary sequence solves \eqref{eq:Problem}. Note that the existence of a strong-fixed points, $x$, implies $P_Q(x)\cap H\neq\emptyset$ (see \cite{BNlocal}).

 Theorem~\ref{th:Qcompact} complements this result by showing that, whenever $Q\cap H\neq\emptyset$, the auxiliary sequence always has at least one cluster point, regardless of the behavior of the Douglas--Rachford sequence. \qede
\end{example}

\begin{example}[Knapsack lower bound feasibility]
 The classical \emph{0-1 knapsack problem} is the binary program
 $$\min\left\{\langle c,x\rangle\mid x\in\{0,1\}^m,\,\langle a,x\rangle\leq b\right\},$$for non-negative vectors $a,c\in\R^m_+$ and a non-negative real number $b\in\R_+$.

 The \emph{0-1 knapsack lower-bound feasibility problem} is a case of~\eqref{eq:Problem} with constraints
 $$H:=\{x\in\R^m\mid \langle a,x\rangle\leq b\},\qquad Q:=\{x\in\{0,1\}^m\mid\langle c,x\rangle\geq \lambda \},$$
where $\lambda\in\R_+$. As a decision problem it is NP-complete \cite[I.5 Corollary~6.11]{NW}.

 Applied to this problem, Theorem~\ref{th:finiteQ} shows that the Douglas--Rachford method either finds a solution in finitely many iterations, or none exists and the norm of the Douglas--Rachford sequence diverges to infinity. Note that, in general, $P_Q$ usually cannot be computed efficiently.\qede
\end{example}

\begin{example}
 Consider the problem of minimizing a linear function $f=\langle a,\cdot\rangle$ over a (weakly) compact set $Q$ for which $P_Q$ is effectively computable or approximable. Problem~\eqref{eq:Problem} is a useful relaxation of this minimization when an upper-bound $b$ on the optimal value is known. Theorem~\ref{th:Qcompact} then applies. \qede
\end{example}

\begin{example}[A sphere and a half-space]
 Application of the Douglas--Rachford method to the problem of finding a point in the intersection of a sphere and an affine line (more generally an affine subspace) was originally investigated in \cite{BSabense,ABglobal}, with global convergence later proven using a Lyapunov stability argument in \cite{benoist}. Here we consider the case in which the affine line is replaced with a half-space.

 Let $Q$ be the unit sphere and $H$ a half-space in $\R^2$. By symmetry, we may assume $a=(0,1)$. Let $x_0\neq 0$ with $x_0(1)>0$. Then $x_k(1)>0$ and $q_k=\frac{x_k}{\|x_k\|}$ for all $k\in\N$, and the iteration given by \eqref{eq:DRiteration} becomes
  $$x_{k+1}(1)=\frac{x_k(1)}{\|x_k\|},\quad
    x_{k+1}(2)=\begin{cases}
                \frac{x_k(2)}{\|x_k\|}                   & \text{if }\left(\frac{2}{\|x_k\|}-1\right)x_k(2)\leq b, \\
                \left(1-\frac{1}{\|x_k\|}\right)x_k(2)+b & \text{if }\left(\frac{2}{\|x_k\|}-1\right)x_k(2)> b. \\
               \end{cases}$$

 If $Q\cap H\neq\emptyset$ (or equivalently $b\geq -1$) then Theorem~\ref{th:Qcompact} below ensures $d(q_k,H)\to 0$. It then follows that either $q_{k_0}\in H\cap Q$ for some $k_0\in\N$ ({\em i.e.,} a solution can be found in finitely many iterations), or $q_k(2)\to b$ and hence $q_k \to (\sqrt{1-b^2},b)\in Q\cap H$. \qede
\end{example}

\section{Properties of the Douglas--Rachford Operator}\label{sec:properties}

 In this section we investigate the behavior of the Douglas--Rachford operator \eqref{eq:DR} and the corresponding iteration without imposing any additional assumptions on the closed set $Q$.

 We begin by distinguishing two cases depending on whether the point is contained in the half-space. Our first proposition shows the image of $H$ under the Douglas--Rachford operator to be a subset of $H$.

\begin{proposition}\label{prop:1}
If $x\in H$ then $\DR(x)\subset H$.
\end{proposition}

\begin{proof}
Choose any $q\in P_Q(x)$. We will distinguish two cases, depending on whether $q$ belongs to $H$. If $q\not\in H$, then
$$\langle a,2q-x\rangle=2(\langle a,q\rangle-b)+(b-\langle a,x\rangle)+b>b,$$
whence, $\DR(x,q)=q+(\langle a,x\rangle+b-2\langle a,q\rangle)a$. Thus,
$$\langle a,\DR(x,q)\rangle=\langle a,q\rangle+\langle a,x\rangle+b-2\langle a,q\rangle\leq 2b-\langle a,q\rangle<b,$$
and we have $\DR(x,q)\in H$.

Suppose now that $q\in H$. If $\langle a,2q-x\rangle\leq b$ we have $\DR(x,q)=q\in H$. Otherwise, if $\langle a,x\rangle+b<2\langle a,q\rangle$, we have $\DR(x,q)=q+(\langle a,x\rangle+b-2\langle a,q\rangle)a$ and
$$\langle a,\DR(x,q)\rangle=\langle a,x\rangle+b-\langle a,q\rangle<\langle a,q\rangle\leq b.$$
Thus, $\DR(x,q)\in H$.
\end{proof}

 Our second proposition characterizes behavior of the Douglas--Rachford operator for points which lie outside the half-space.

\begin{proposition}\label{prop:2}
Suppose that $x\not\in H$ and $q\in P_Q(x)$. The following holds:
\begin{enumerate}[(i)]
    \setlength{\itemsep}{0pt}
    \setlength{\parskip}{0pt}
    \setlength{\parsep}{0pt}
\item if $q\in H$, then $\DR(x,q)=q$;
\item if $q\not\in H$, then
    \begin{enumerate}[(a)]
    \item if $d(x,H)\geq2d(q,H)$, then $\DR(x,q)=q$ and $\DR(\DR(x,q))=P_L(q)$;
    \item if $d(x,H)< 2d(q,H)$, then $\DR(x,q)=q+(\langle a,x\rangle+b-2\langle a,q\rangle)a$, and
        \begin{enumerate}[(I)]
        \item if $d(x,H)\leq d(q,H)$, then $\DR(x,q)\in H$;
        \item if $d(x,H)> d(q,H)$, then $d(\DR(x,q),H)=d(x,H)-d(q,H)$. Furthermore, if $q\in P_Q(\DR(x,q))$, then $\DR(\DR(x,q),q)\in H$.
        \end{enumerate}
    \end{enumerate}
\end{enumerate}
\end{proposition}

\begin{proof}
\emph{(i)} If $q\in H$, then $$\langle a,2q-x\rangle=2(\langle a,q\rangle-b)+b+(b-\langle a,x\rangle)< b;$$
 whence, $\DR(x,q)=q$, as claimed.

\emph{(ii)(a)} We have $d(x,H)\geq2d(q,H)\iff \langle a,2q-x\rangle\leq b$. Thus, $\DR(x,q)=q$ and
$$\DR(\DR(x,q))=\DR(q)=q-(\langle a,q\rangle-b)a=P_L(q),$$
as required.

\emph{(ii)(b)(I)} If $\langle a,x\rangle-b=d(x,H)\leq d(q,H)=\langle a,q\rangle-b$, then
$$\langle a,\DR(x,q)\rangle=\langle a,q\rangle+\langle a,x\rangle+b-2\langle a,q\rangle=\langle a,x-q\rangle+b\leq b,$$
whence, $\DR(x,q)\in H$.

\emph{(ii)(b)(II)} If $d(x,H)> d(q,H)$, then
$$\langle a,\DR(x,q)\rangle=\langle a,x-q\rangle+b>b,$$
and
$$d(\DR(x,q),H)=\langle a,x-q\rangle=(\langle a,x\rangle-b)-(\langle a,q\rangle-b)=d(x,H)-d(q,H).$$

Further, suppose that $q\in P_Q(\DR(x,q))$. Then,
\begin{align*}
\langle a,2{q}-\DR(x,q)\rangle&=\langle a,2q-q-(\langle a,x\rangle+b-2\langle a,q\rangle)a\rangle\\
&=\langle a,q\rangle-(\langle a,x\rangle-b)+2(\langle a,q\rangle-b)\\
&=\langle a,q\rangle-d(x,H)+2d(q,H)>b.
\end{align*}
Hence,
\begin{align*}
\DR(\DR(x,q),q)&=q+\left(\langle a,q+\left(\langle a,x\rangle+b-2\langle a,q\rangle\right)a\rangle+b-2\langle a,q\rangle\right)a\\
&=q+(\langle a,x\rangle-3\langle a,q\rangle+2b)a.
\end{align*}
Finally, we have
\begin{align*}
\langle a,\DR(\DR(x,q),q)\rangle&=(\langle a,x\rangle-b)-2(\langle a,q\rangle-b)+b\\
&=d(x,H)-2d(q,H)+b<b;
\end{align*}
whence, $\DR(\DR(x,q),q)\in H$.
\end{proof}

 By combining Propositions~\ref{prop:1} \& \ref{prop:2}, we shall deduce the following lemma   concerning the behavior of Douglas--Rachford iterations which never enter the half-space.

\begin{lemma}\label{lem:qkto0}
Let $\{x_k\}$ be a Douglas--Rachford sequence with auxiliary sequence $\{q_k\}$. If $x_k\notin H$ for each $k\in\mathbb{N}$, then $q_k\not\in H$ for all $k\in\mathbb{N}$, the sequence $\{d(x_k,L)\}$ is strictly monotone decreasing,  and $$\lim_{k\to\infty}d(x_k,L)=\lim_{k\to\infty} d(q_k,L)=0.$$
\end{lemma}
\begin{proof}
 If $x_k\notin H$ for all $k$, by Propositions~\ref{prop:1} \& \ref{prop:2}, it must be that $d(q_k,H)<d(x_k,H)<2d(q_k,H)\text{ and }d(x_{k+1},H)=d(x_k,H)-d(q_k,H).$
 Thus the sequence $\{d(x_k,H)\}$ is strictly decreasing and bounded below by zero. Since $\lim_{k\to\infty} d(x_k,H)$ exists, we deduce
$$\lim_{k\to\infty}d(q_k,H)=\lim_{k\to\infty}\left[d(x_k,H)-d(x_{k+1},H)\right]=0.$$
 Thus $\lim_{k\to\infty} d(q_k,L)=\lim_{k\to\infty} d(q_k,H)=0$, and, as $d(x_k,H)<2d(q_k,H)$, we have $\lim_{k\to\infty}d(x_k,L)=\lim_{k\to\infty}d(x_k,H)=0$.
\end{proof}

 We now turn our attention to the precise structure of the Douglas--Rachford operator at points which lie within the half-space. The following proposition gives a relationship between consecutive terms in a Douglas--Rachford sequence.

\begin{proposition}\label{prop:3}
For any $x\in H$ and any $q\in P_Q(x)$ such that $q\not\in H$ one has
\begin{align*}
\DR(x,q)&=q-(d(x,L)+2d(q,L))a,\\
d(\DR(x,q),L)&=d(q,L)+d(x,L).
\end{align*}
\end{proposition}

\begin{proof}
Since $x\in H$ and $q\not\in H$ we have that
$\langle a,q\rangle>b\geq \langle a,x\rangle$. Thus,
$$\langle a,2q-x\rangle=\langle a,q\rangle+\langle a,q-x\rangle>b,$$
and we have
\begin{align*}
\DR(x,q)&=q-(-\langle a,x\rangle-b+2\langle a,q\rangle)a=q-(b-\langle a,x\rangle+2(\langle a,q\rangle-b))\\
&=q-(d(x,L)+2d(q,L))a.
\end{align*}
Then,
$$d(\DR(x,q),L)=\left|\langle a,q\rangle-b-d(x,L)-2d(q,L)\right|=d(q,L)+d(x,L),$$
which completes the proof.
\end{proof}

 When the Douglas--Rachford point $x$ lies in $H$, our next proposition relates $x$ to auxiliary points which lie in $Q\setminus H$.

\begin{proposition}\label{prop:4}
Let $x\in H$ and $q\in P_Q(x)\setminus H$. Then, if $p\in P_Q(\DR(x,q))\setminus H$, with $p\neq q$, one has
\begin{equation}\label{eq:prop4_inequality}
d(p,H)+\|\DR(x,q)-q\|\leq d(q,H)+d(\DR(x,q),Q).
\end{equation}
Furthermore, we have $d(p,H)<d(q,H)$.
\end{proposition}
\begin{proof}
Let $z:=\DR(x,q)\in H$. Since $p\in P_Q(z)$, we have $\|z-q\|\geq \|z-p\|$.
Observe that
\begin{equation}\label{eq:1}
0< d(p,H)+d(z,L)=\langle a,p\rangle-b+b-\langle a,z\rangle=\langle a,p-z\rangle\leq \|p-z\|.
\end{equation}
If the inequality in~\eqref{eq:1} is not strict, there is some $\lambda\in\R$ such that $\lambda a=p-z$. Then,
$$\langle a,p-z\rangle=\|p-z\|\implies \lambda = \lambda \|a\|^2 = \|p-z\|>0.$$
If $\|z-q\|=\|z-p\|=\lambda$, using Proposition~\ref{prop:3}, we obtain a contradiction with the assumption that $p\neq q$:
\begin{align*}
\|p-q\|&=\|\lambda a+z-q\|=\|\lambda a-(d(x,L)+2d(q,L))a\|=|\lambda-\|z-q\||=0.
\end{align*}
Thus, $\|z-q\|>\|z-p\|=\lambda$, and by Proposition~\ref{prop:3} we have
\begin{align}
d(p,H)-d(q,H)&=\langle a,p-q\rangle=\langle a,\lambda a+z-q\rangle\nonumber\\
&=\lambda +\langle a,z-q\rangle=\lambda-\left(d(x,L)+2d(q,L)\right)\nonumber\\
&=\lambda-\|z-q\|<0,\label{eq:lastpart}
\end{align}
 whence,
  $$d(p,H)+\|z-q\|=d(q,H)+\|z-p\|=d(q,H)+d(z,Q),$$
 and we are done.

Otherwise, suppose that the inequality in~\eqref{eq:1} is strict, i.e.,
\begin{equation}\label{eq:3}
d(p,H)+d(z,L)< \|p-z\|.
\end{equation}
By Proposition~\ref{prop:3} we know that
$$\|z-q\|=d(x,L)+2d(q,L)\quad\text{and}\quad d(z,L)=d(q,L)+d(x,L);$$
whence,
$$\|z-q\|=d(q,L)+d(z,L).$$
Thus, by~\eqref{eq:3},
$$d(q,H)=d(q,L)=\|z-q\|-d(z,L)> \|z-q\|-\|z-p\|+d(p,H)\geq d(p,H),$$
which proves~\eqref{eq:prop4_inequality}. The last assertion in the statement follows  from~\eqref{eq:lastpart} and  the inequality above.
\end{proof}

\begin{remark}
 In particular, Proposition~\ref{prop:4} shows that once a Douglas--Rachford sequence enters the half-space, it is not possible for its auxiliary sequence to cycle within points from $Q\setminus H$. Indeed, once an element $q_{k_0}\in Q\setminus H$ appears as the $k_0$-th term auxiliary sequence either the sequence  remains constant with $q_k=q_{k_0}$ for all $k\geq k_0$, or  there exists $k_1\geq k_0$ such that $q_{k_0}\not\in P_Q(x_{k})$ for all $k\geq k_1$.\qede
\end{remark}

\section{Analysis of the Algorithm}\label{sec:analysis}
 In this section we establish our main results which analyze the global convergence properties of the Douglas--Rachford method assuming some additional structure on the closed set $Q$. These assumption encompass, but are not limited to, the setting in which $Q$ is a finite set, compact set, or a weakly compact set.

 Consider the set $Q\setminus H$ equipped with the binary relation $\lesssim$ defined, for all $p,p'\in Q\setminus H$, by
  $$p\lesssim p'\iff d(p,H)\leq d(p',H).$$
 Since this relation is reflexive and transitive, it defines a \emph{quasi-ordering} on the set $Q\setminus H$.

  We first consider the case in which the following assumption holds.

 \begin{assumption}\label{ass:qwo}
  The relation $\lesssim$ is a \emph{well-quasi-ordering} on $Q\setminus H$. That is, any sequence $\{z_k\}$ contained in $Q\setminus H$ contains a pair of terms such that $z_i\lesssim z_j$ and $i<j$. \qede
 \end{assumption}

 \begin{remark} Assumption~\ref{ass:qwo} is equivalent to requiring that there exist no sequence $\{p_k\}$ with $p_k\in Q\setminus H$, for all $k\in\N$, such that $\{d(p_k,H)\}$ is a strictly monotone decreasing sequence.\qede
 \end{remark}

 The following lemma shows that, under this assumption, any Douglas--Rachford sequence eventually enters the half-space.

\begin{lemma}\label{lem:DRentersH}
 Suppose Assumption~\ref{ass:qwo} holds. Then any Douglas--Rachford sequence $\{x_k\}$ enters and remains in $H$ after a finite number of steps. That is, there exists $k_0\in\mathbb{N}$ such that $x_k\in H$ for all $k\geq k_0$.
\end{lemma}
\begin{proof}
 Let $\{x_k\}$ be a Douglas--Rachford sequence with auxiliary sequence $\{q_k\}$ such that $x_k\not\in H$ for all $k\in\mathbb{N}$. By Proposition~\ref{prop:1} \& Proposition~\ref{prop:2}, $q_k\not\in H$ for all $k\in\mathbb{N}$, and
  $$d(x_{k+1},H)=d(x_k,H)-d(q_k,H)\quad\text{for all }k\in\mathbb{N}.$$
 By telescoping we obtain
  $$0<d(x_{k+1},H)=d(x_1,H)-\sum_{j=1}^kd(q_j,H).$$
 We therefore deduce that $\sum_{j=1}^kd(q_j,H)<d(x_1,H)$ for all $k\in\N$, hence the sum is convergent, and in particular $d(q_k,H)\to 0$. For all $k\in\N$, $q_k\not\in H$ hence $d(q_k,H)>0$ for all $k\in\N$. Thus there exists a subsequence $\{q_{k_j}\}$ such that $\{d(q_{k_j},H)\}$ is strictly monotone decreasing. This contradicts Assumption~\ref{ass:qwo} and completes the proof.
\end{proof}

 The following example shows that the conclusions of Lemma~\ref{lem:DRentersH} need not hold without something akin to Assumption~\ref{ass:qwo}.

\begin{example}\label{ex:infiniteCompactDR}
 Consider the following subsets of the real line:
  $$Q:=\left\{\frac{2}{3^k} \;\middle|\; k=0,1,2,\dots\right\}\cup\left\{0\right\}, \quad H:=\left\{x\in\R\mid x\leq 0\right\}.$$
 For initial point $x_0=1$, the Douglas--Rachford iteration and auxiliary sequence are
  \begin{equation*}
   x_k=\frac{1}{3^k},\qquad q_k:=P_Q(x_k)=\frac{2}{3^{k+1}}.
  \end{equation*}
 Both $\{x_n\}$ and $\{q_n\}$ are positive real numbers, and hence never enter $H$. \qede
\end{example}

 We shall also require to following result regarding Assumption~\ref{ass:qwo}.

\begin{lemma}\label{lem:qkconstant}
 Suppose Assumption~\ref{ass:qwo} holds. Let $\{x_k\}$ be a Douglas--Rachford sequence with auxiliary sequence $\{q_k\}$ such that $q_k\not\in H$ for all $k\in\mathbb{N}$. Then the auxiliary sequence is eventually constant. That is, there exists $k_1\in\N$ such that $q_{k_1}=q_k$ for all $k\geq k_1$.
\end{lemma}
\begin{proof}
 Suppose that the auxiliary sequence $\{q_k\}$ is not eventually constant.
 By Lemma~\ref{lem:DRentersH}, we may assume, without loss of generality, that $x_k\in H$ for all $k\in\N$.
 Using Proposition~\ref{prop:4} and the fact that $\{q_k\}$ is not eventually constant, we deduce the existence of a subsequence $\{q_{k_j}\}$ such that $\{d(q_{k_j},H)\}$ is strictly monotone decreasing. This contradicts Assumption~\ref{ass:qwo} and completes the proof.
\end{proof}

 We formulate now our first main result regarding convergence of the method under Assumption~\ref{ass:qwo}.

\begin{theorem}\label{th:qwo}
 Suppose Assumption~\ref{ass:qwo} holds. Let $\{x_k\}$ be a Douglas--Rachford sequence with auxiliary sequence $\{q_k\}$. Then either: (i) there exists $k_0\in\N$ such that $q_{k_0}\in Q\cap H\neq\emptyset$, or (ii) $H\cap Q=\emptyset$. Moreover, in the latter case, $\|x_k\|\to+\infty$.
\end{theorem}
\begin{proof}
 Suppose $q_k\not\in H$ for all $k\in\N$.
By Lemma~\ref{lem:DRentersH}, there is some $k_0\in\mathbb{N}$ such that $x_k\in H$ for all $k\geq k_0$.  By Lemma~\ref{lem:qkconstant}, there is some $k_1\in\N$ with $k_1\geq k_0$ and some $q\in Q\setminus H$ such that $P_Q(x_k)=q$ for all $k\geq k_1$. Then, by Proposition~\ref{prop:3},
\begin{align}
x_{k+1}&=q-(d(x_{k},L)+2d(q,L))a=q-(d(x_{k-1},L)+3d(q,L))a\nonumber\\
&=\ldots=q-(d(x_{k_1},L)+(2+k-k_1)d(q,L))a,\label{eq:4}
\end{align}
for all $k\geq k_1$. Hence,
 \begin{equation}
  \|x_{k+1}\|\geq d(x_{k_1},L)+(2+k-k_1)d(q,L)-\|q\|, \label{eq:div1}
 \end{equation}
whence, $\|x_k\|\to\infty$. Further, suppose by contradiction that there is $p\in H\cap Q$. For any $k\geq k_1$,  by~\eqref{eq:4}, one has
\begin{align*}
\|x_{k+1}-p\|^2&=\|(x_{k+1}-q)+(q-p)\|^2\\
&=\|x_{k+1}-q\|^2+\|q-p\|^2+2\langle x_{k+1}-q,q-p\rangle\\
&=\|x_{k+1}-q\|^2+\|q-p\|^2-2\langle a,q-p\rangle(d(x_{k_1},L)+(2+k-k_1)d(q,L)).
\end{align*}
Observe that, since $q\not\in H$ and $p\in H$, we have
$$\langle a,q-p\rangle=(\langle a,q\rangle-b)+(b-\langle a,p\rangle)>0.$$
Thus, there is some $k_2\in\N$, with $k_2\geq k_1$, such that
$$\|q-p\|^2-2\langle a,q-p\rangle(d(x_{k_1},L)+(2+k-k_1)d(q,L))<0,\quad\forall k\geq k_2.$$
Hence,
$$\|x_{k+1}-p\|^2<\|x_{k+1}-q\|^2,\quad\forall k\geq k_2,$$
that is, we obtain a contradiction  with the fact that $P_Q(x_{k+1})=q$.
\end{proof}

Theorem~\ref{th:qwo} suggests the following algorithm for finding a point $Q\cap H$.\newline

\RestyleAlgo{boxruled}
\begin{algorithm}[H]
 \textbf{Input:} $x_{0}\in\Hilbert$\;
 Choose any $q_0\in P_Q(x_0)$\;
 Set $k:=0$\;
 \While{$q_k\not\in H$}{
      $x_{k+1}:=\DR(x_{k},q_k)$\;
      Choose any $q_{k+1}\in P_Q(x_{k+1})$\;$k:=k+1$\;}
 \textbf{Output:} $q_k\in Q\cap H$\;
\caption{The Douglas--Rachford algorithm for solving \eqref{eq:Problem} with Assumption~\ref{ass:qwo}.}\label{alg:qwo}
\end{algorithm}

 \vspace{1ex}

 The following corollary provides justification for Algorithm~\ref{ass:qwo}.

\begin{corollary}
 Suppose Assumption~\ref{ass:qwo} holds. Then Algorithm~\ref{alg:qwo} either: (i) terminates finitely to a point in $H\cap Q$, or (ii) $Q\cap H=\emptyset$ and $\|x_k\|\to+\infty$.
\end{corollary}
\begin{proof}
 Follows directly from Theorem~\ref{th:qwo}.
\end{proof}

 It is easily seen that any finite set satisfies Assumption~\ref{ass:qwo}, hence Theorem~\ref{th:qwo} applies. However, as we now show, a stronger result holds in this special case. We first require the following two lemmata.

\begin{lemma}\label{lem:qk in H}
 Let $\{x_k\}$ be a Douglas--Rachford sequence with auxiliary sequence $\{q_k\}$. Suppose there exists $k_0\in\N$ such that $x_{k_0}\in H$ and $q_{k_0}\in H$. Then $q_k\in H$ and  $d(q_{k+1},L)\geq d(q_k,L)$  for all $k\geq k_0$. Furthermore, for any $k\geq k_0$, one has $d(q_{k+1},L)= d(q_k,L)$ if and only if $q_{k+1}=q_{k}$.
\end{lemma}
\begin{proof}
 We distinguish two cases. First suppose $\langle a,2q_{k_0}-x_{k_0}\rangle\leq b$. By \eqref{eq:DRiteration} we have $x_{k_0+1}=q_{k_0}$ and hence $q_{k_0+1}=q_{k_0}$.

 Now suppose $\langle a,2q_{k_0}-x_{k_0}\rangle> b$. By \eqref{eq:DRiteration},
\begin{equation}\label{eq:xk0plus1}
x_{k_0+1}=q_{k_0}-(\langle a,2q_{k_0}-x_{k_0}\rangle-b)a=q_{k_0}-d(2q_{k_0}-x_{k_0},H)a.
\end{equation}
  Then, as $q_{k_0+1}\in P_Q(x_{k_0+1})$, we have
  \begin{align}
  \langle a,q_{k_0+1}\rangle&=\langle a,q_{k_0+1}-x_{k_0+1}\rangle+\langle a,x_{k_0+1}\rangle\nonumber\\
  &\leq\|q_{k_0+1}-x_{k_0+1}\|+\langle a,q_{k_0}\rangle-d(2q_{k_0}-x_{k_0},H)\label{eq:ineq1}\\
  &\leq \|q_{k_0}-x_{k_0+1}\|+\langle a,q_{k_0}\rangle-d(2q_{k_0}-x_{k_0},H)\label{eq:ineq2}\\
  &= \langle a,q_{k_0}\rangle\leq b\nonumber,
  \end{align}
 whence $q_{k_0+1}\in H$ and $d(q_{k_0+1},L)\geq d(q_{k_0},L)$.

 Further, let us assume that $d(q_{k_0+1},L)=d(q_{k_0},L)$. Then, we have $\langle a,q_{k_0+1}\rangle=\langle a,q_{k_0}\rangle$. Therefore the inequalities in~\eqref{eq:ineq1} and~\eqref{eq:ineq2} must be equalities, from where we deduce
 \begin{equation}\label{eq:3equalities}
 \langle a,q_{k_0+1}-x_{k_0+1}\rangle=\|q_{k_0+1}-x_{k_0+1}\|=\|q_{k_0}-x_{k_0+1}\|.
 \end{equation}
 Hence, by~\eqref{eq:3equalities} and~\eqref{eq:xk0plus1}, we have
 \begin{align*}
 \langle q_{k_0+1}-x_{k_0+1},x_{k_0+1}-q_{k_0}\rangle&=\langle q_{k_0+1}-x_{k_0+1},-d(2q_{k_0}-x_{k_0},H)a\rangle\\
 &=-\|x_{k_0+1}-q_{k_0}\|\langle q_{k_0+1}-x_{k_0+1},a\rangle\\
 &=-\|x_{k_0+1}-q_{k_0}\|^2.
 \end{align*}
 Thus, using again~\eqref{eq:3equalities}, we obtain
 \begin{align*}
 \|q_{k_0+1}-q_{k_0}\|^2&=\|q_{k_0+1}-x_{k_0+1}\|^2+\|x_{k_0+1}-q_{k_0}\|^2+2\langle q_{k_0+1}-x_{k_0+1},x_{k_0+1}-q_{k_0}\rangle\\
 &= \|q_{k_0+1}-x_{k_0+1}\|^2-\|x_{k_0+1}-q_{k_0}\|^2=0;
 \end{align*}
 that is, $q_{k_0+1}=q_{k_0}$. The result now follows by induction.
\end{proof}

\begin{lemma}\label{lem:xkqk in H}
 Let $\{x_k\}$ be a Douglas--Rachford sequence with auxiliary sequence $\{q_k\}$.
 Suppose there exists $k_0\in\N$ such that $x_{k}\in H$ and $q_k=q\in H$ for all $k\geq k_0$. Then there exists $k_1\in \N$ with $k_1\geq k_0$ such that $x_k=x_{k_1}$ for all $k\geq k_1$.
\end{lemma}

\begin{proof}
If $\langle a,2q-x_j\rangle\leq b$ for some $j\geq k_0$, then by~\eqref{eq:DRiteration} we deduce $x_{k}=q$ for all $k\geq j+1$ and we are done.

Otherwise, by~\eqref{eq:DRiteration}, we have $x_{k+1}=q+(\langle a,x_k\rangle+b-2\langle a,q\rangle)a$ for all $k\geq k_0$. Thus,
  \begin{align*}
   \langle a,2q-x_{k+1}\rangle &= \left\langle a,2q-(q+(\langle a,x_k\rangle+b-2\langle a,q\rangle)a)\right\rangle \\
                   &=\langle a,2q-x_k\rangle - (b-\langle a,q\rangle) \\
                   &= \langle a,2q-x_k\rangle-d(q,L),
  \end{align*}
  whence,
  $$b<\langle a,2q-x_k\rangle = \langle a,2q-x_{k_0}\rangle-(k-k_0)d(q,L).$$
This implies $d(q,L)=0$ and then,
$$x_{k+1}=q+(\langle a,x_k\rangle+b-2\langle a,q\rangle)a=q+(\langle a,x_k\rangle -b)a,$$
from which we deduce $x_{k}=q+(\langle a,x_k\rangle -b)a$ for all $k\geq k_0+1$. This completes the proof.
\end{proof}

 The following theorem is a refinement of Theorem~\ref{th:qwo} when $Q$ is assumed to be finite.

\begin{theorem}\label{th:finiteQ}
  Suppose $Q$ is finite. Let $\{x_k\}$ be a Douglas--Rachford sequence with auxiliary sequence $\{q_k\}$. Then either: (i) $\{x_k\}$ and $\{q_k\}$ are eventually constant and the limit of $\{q_k\}$ is contained in $H\cap Q\neq\emptyset$, or (ii) $H\cap Q=\emptyset$ and $\|x_k\|\to+\infty$.
\end{theorem}
\begin{proof}
 By Theorem~\ref{th:qwo}, it suffices to show that the sequences $\{x_k\}$ and $\{q_k\}$ are eventually constant and the limit of $\{q_k\}$ is contained in $H\cap Q\neq\emptyset$, assuming there exists a $k_0\in\N$ such that $x_{k_0}\in H$ and $q_{k_0}\in H$.

 To this end, suppose there is some $k_0\in\N$ such that $x_{k_0}\in H$ and $q_{k_0}\in H$. By Proposition~\ref{prop:1} and Lemma~\ref{lem:qk in H} we have $x_k\in H, q_k\in H$ and $d(q_{k+1},L)\geq d(q_k,L)$ with equality if and only if $q_{k+1}=q_k$, for all $k\geq k_0$. Since $Q$ is finite, the latter implies that there exists $k_1\geq k_0$ such that $q_{k}=q_{k_1}$ for all $k \geq k_1$. By Lemma~\ref{lem:xkqk in H}, there exists a $k_2\geq k_1$ such $x_{k}=x_{k_2}$ for all $k\geq k_2$. This completes the proof.
\end{proof}

 Define the mapping
  $$\mathcal{Q}(\cdot):=\{p\in Q \mid d(p,H)\leq d(\cdot,H)\}.$$

 We now consider the case in which one of the following assumptions hold. In particular, they include the cases in which $Q$ is (weakly) compact.

\begin{assumption}\label{ass:compactQH}
 The function $\iota_Q+d(\cdot,H)$ has compact lower-level sets. In particular, for every $q\in Q$, the set $\mathcal{Q}(q)$ is compact.\qede
\end{assumption}

\begin{assumption}\label{ass:wcompactQH}
  The function $\iota_Q+d(\cdot,H)$ has weakly compact lower-level sets. In particular, for every $q\in Q$, the set $\mathcal{Q}(q)$ is weakly compact. \qede
\end{assumption}

In finite dimension, Assumptions~\ref{ass:compactQH} \& \ref{ass:wcompactQH} may also be written in the form $$\lim_{x\in Q,\|x\|\to\infty} d(x,H)=+\infty.$$

 The following is our final main result. It characterizes behavior of the algorithm for $Q$ satisfying Assumption~\ref{ass:compactQH} (respectively Assumption~\ref{ass:wcompactQH}). In particular, it shows that the Douglas--Rachford algorithm can be used to determine consistency of the feasibility problem \eqref{eq:Problem} and to find a solution when such exist. Since the proofs are similar, we prove the result under both assumptions simultaneously.

\begin{theorem}\label{th:Qcompact}
 Suppose Assumption~\ref{ass:compactQH} holds (or Assumption~\ref{ass:wcompactQH} holds). Let $\{x_k\}$ be a Douglas--Rachford sequence with auxiliary sequence $\{q_k\}$.  Then either: (i) $d(q_k,H)\to 0$ and the set of (weak) cluster points of the auxiliary sequence is non-empty and contained in $Q\cap H$, or (ii) $d(q_k,H)\to\beta$ for some $\beta>0$ and $H\cap Q=\emptyset$. Moreover, in the latter case, $\|x_k\|\to+\infty$.
\end{theorem}
\begin{proof}
Let $\{x_k\}$ be a Douglas--Rachford sequence with auxiliary sequence $\{q_k\}$. We distinguish two cases.

 First suppose that $x_k\not\in H$ for all $k$. By Proposition~\ref{prop:2}, $q_k\not\in H$ for all $k$ and, by Lemma~\ref{lem:qkto0}, $d(q_k,H)\to 0$. Hence there exists ${k_0}\in\N$ such $d(q_k,H)\leq d(q_{k_0},H)$ for all $k\geq k_0$, whence $q_k\in\mathcal{Q}(q_{k_0})$ for all $k\geq k_0$.  Since $\mathcal{Q}(q_{k_0})$ is (weakly) compact and $d(\cdot,H)$ is (weakly) continuous, the set of (weak) cluster points of the auxiliary sequence $\{q_{k}\}$ is non-empty and contained in $Q\cap H$.

 Next suppose there is some $k_0\in\mathbb{N}$ such that $x_{k_0}\in H$. Then, by Proposition~\ref{prop:1}, we have $x_k\in H$ for all $k\geq k_0$. On one hand, if there exists $k_1\geq k_0$ such that $q_{k_1}\in H$ then, by Lemma~\ref{lem:qk in H}, $q_k\in H$ for all $k\geq k_1$. Since $\mathcal{Q}(q_{k_1})=Q\cap H$ is (weakly) compact, it follows that the set of (weak) cluster points of the auxiliary sequence $\{q_{k}\}$ is non-empty and contained in $Q\cap H$, and we are done.

 On the other hand, suppose $q_k\not\in H$ for all $k\geq k_0$. Then by Proposition~\ref{prop:4}, $q_k\in\mathcal{Q}(q_{k_0})$ for all $k\geq k_0$,  and $\{d(q_k,H)\}_{k=k_0}^\infty$ is monotone decreasing and bounded below by zero, hence
  $$\beta:=\inf_{k\geq k_0}d(q_k,H)=\lim_{k\to\infty}d(q_k,H)\geq 0.$$
By Proposition~\ref{prop:3}, for all $k\geq k_0$,
\begin{align*}
x_{k+1}&=q_k-(d(x_{k},L)+2d(q_k,L))a\nonumber\\
&=q_k-(d(x_{k-1},L)+d(q_{k-1},L)+2d(q_k,L))a\nonumber\\
&=\ldots=q_k-\left(d(x_{k_0},L)+d(q_k,L)+\sum_{j=k_0}^kd(q_j,L)\right)a.
\end{align*}
For all $k\geq k_0$, we may therefore express $x_{k+1}=q_k-\lambda_ka$ where
 \begin{equation}
  \lambda_k:=d(x_{k_0},L)+d(q_k,L)+\sum_{j=k_0}^kd(q_j,L)\geq (k-k_0)\beta\geq 0. \label{eq:eq6}
 \end{equation}
 If $\beta=0$ then, by the (weak) compactness of $\mathcal{Q}(q_{k_0})$ and (weak) continuity of $d(\cdot,H)$, the set of (weak) cluster points of the auxiliary sequence $\{q_{k}\}$ is non-empty and contained in $Q\cap H$.
 	
 Conversely, assume $\beta>0$.
 Since it is (weakly) compact, the set $\mathcal{Q}(q_{k_0})$ is bounded, hence there exists $K>0$ such that $\|q\|\leq K$ for all $q\in \mathcal{Q}(q_{k_0})$, and thus
   \begin{equation}
    \|x_{k+1}\|\geq \lambda_k-\|q_k\|\geq \lambda_k-K\geq (k-k_0)\beta-K\to +\infty.\label{eq:7}
   \end{equation}
 To complete the proof, we must show $Q\cap H=\emptyset$. To this end, suppose there is a $p\in Q\cap H$. We claim that, for all $k\geq k_0$,
  \begin{equation*}
   \|x_k-q_k\|>\left\|x_k-P_L(q_k)-\frac{\beta}{2}a\right\|.
  \end{equation*}
To prove this claim, first observe
\begin{align*}
\left(\langle a,q_k\rangle-b-\frac{\beta}{2}\right)+2\langle a,x_k-q_k\rangle&=-\left(\langle a,q_k\rangle-b\right)-\frac{\beta}{2}-2\left(b-\langle a,x_k\rangle\right)\\
&=-d(q_k,L)-\frac{\beta}{2}-2d(x_k,L)<0.
\end{align*}
Since $\left(\langle a,q_k\rangle-b-\frac{\beta}{2}\right)=d(q_k,L)-\frac{\beta}{2}\geq\beta-\frac{\beta}{2}=\frac{\beta}{2}>0$,
we deduce the claimed result
\begin{align*}
\left\|x_k-P_L(q_k)-\frac{\beta}{2}a\right\|^2&=\left\|x_k-q_k+\left(\langle a,q_k\rangle-b-\frac{\beta}{2}\right)a\right\|^2\\
&=\|x_k-q_k\|^2+\left(\langle a,q_k\rangle-b-\frac{\beta}{2}\right)^2+
2\left(\langle a,q_k\rangle-b-\frac{\beta}{2}\right)\langle a,x_k-q_k\rangle\\
&<\|x_k-q_k\|^2.
\end{align*}
Further, for all $k\geq k_0$, we have
\begin{align*}
\|x_{k+1}-q_{k+1}\|^2&>\left\|x_{k+1}-P_L(q_{k+1})-\frac{\beta}{2}a\right\|^2\\
&=\|x_{k+1}-p\|^2+\left\|p-P_L(q_{k+1})-\frac{\beta}{2}a\right\|^2+2\left\langle x_{k+1}-p,
p-P_L(q_{k+1})-\frac{\beta}{2}a\right\rangle.
\end{align*}
Since
\begin{align*}
\left\langle a,P_L(q_{k+1})+\frac{\beta}{2}a-p\right\rangle&=\left\langle a,q_{k+1}-\left(\langle a,q_{k+1}\rangle-b\right)a+\frac{\beta}{2}a-p\right\rangle\\
&=b-\langle a,p\rangle+\frac{\beta}{2}=d(p,L)+\frac{\beta}{2},
\end{align*}
we obtain
\begin{align*}
\left\langle x_{k+1}-p,p-P_L(q_{k+1})-\frac{\beta}{2}a\right\rangle &=\left\langle q_k-\lambda_k a-p,p-P_L(q_{k+1})-\frac{\beta}{2}a\right\rangle\\
&=\left\langle q_k-p,p-P_L(q_{k+1})-\frac{\beta}{2}a\right\rangle\\
&\quad+\lambda_k \left\langle a,P_L(q_{k+1})+\frac{\beta}{2}a-p\right\rangle\\
&=\left\langle q_k-p,p-P_L(q_{k+1})-\frac{\beta}{2}a\right\rangle+\lambda_k \left(d(p,L)+\frac{\beta}{2}\right).
\end{align*}
Then,
\begin{align*}
\left\|x_{k+1}-q_{k+1}\right\|^2&>\|x_{k+1}-p\|^2+\left\|p-P_L(q_{k+1})-\frac{\beta}{2}a\right\|^2\\
&\quad+2\left\langle q_k-p,p-P_L(q_{k+1})-\frac{\beta}{2}a\right\rangle+2\lambda_k \left(d(p,L)+\frac{\beta}{2}\right).
\end{align*}
By the boundedness of $\mathcal{Q}(q_{k_0})$  and since $H\cap Q\subseteq\mathcal{Q}(q_{k_0})$, we have
$$\eta:=\min_{w,z\in \mathcal{Q}(q_{k_0})}\left\{\left\|p-P_L(z)-\frac{\beta}{2}a\right\|^2+2\left\langle w-p,p-P_L(z)-\frac{\beta}{2}a\right\rangle\right\}\in\R.$$
Therefore,
$$\left\|x_{k+1}-q_{k+1}\right\|^2>\|x_{k+1}-p\|^2+\eta+2\lambda_k\left(d(p,L)+\frac{\beta}{2}\right).$$
Finally, since $\lambda_k\to+\infty$ (by \eqref{eq:eq6}) and $d(p,L)+\frac{\beta}{2}>0$, there is some $k_1\geq k_0$ such that $\eta+2\lambda_k\left(d(p,L)+\frac{\beta}{2}\right)>0$ for all $k\geq k_1$. Then
$$\left\|x_{k+1}-q_{k+1}\right\|^2>\|x_{k+1}-p\|^2,\quad\forall k\geq k_1,$$
which contradicts the fact that $q_{k+1}\in P_Q(x_{k+1})$. We therefore have that $H\cap Q=\emptyset$ and the proof is complete.
\end{proof}

 Every compact set satisfies Assumption~\ref{ass:compactQH}, and thus we deduce the following important corollary.

\begin{corollary}
 Suppose $Q$ is a compact set (or a weakly compact set). Let $\{x_k\}$ be a Douglas--Rachford sequence with auxiliary sequence $\{q_k\}$.  Then either: (i) $d(q_k,H)\to 0$ and the set of (weak) cluster points of the auxiliary sequence is non-empty and contained in $Q\cap H$, or (ii) $d(q_k,H)\to\beta$ for some $\beta>0$ and $H\cap Q=\emptyset$. Moreover, in the latter case, $\|x_k\|\to+\infty$.
\end{corollary}
\begin{proof}
 Follows immediately from Theorem~\ref{th:Qcompact}.
\end{proof}

\begin{remark}[Rate of divergence]
 A closer look at the proofs of Theorem~\ref{th:Qcompact} (resp. Theorem~\ref{th:qwo}) in the case that $H\cap Q=\emptyset$, shows that~\eqref{eq:7} (resp. \eqref{eq:div1}) gives information regarding the rate of divergence of the Douglas--Rachford sequence. More precisely, the sequence diverges with \emph{at least linear rate} in the sense that $\|x_{k+1}\|\geq kM+K$ for some $M>0$ and $K\in\R$. \qede
\end{remark}

\begin{remark}
 In general, the Douglas--Rachford sequence need not converge finitely for compact $Q$ (see Theorem~\ref{th:finiteQ}). Example~\ref{ex:infiniteCompactDR} serves as a counter-example.\qede
\end{remark}

\begin{remark}[The auxiliary sequence]
 During each iteration of the Douglas--Rachford method, the next term in the auxiliary sequence must be selected from the set $P_Q(x_k)$. We note that all the results in this section hold regardless of how these points are chosen.\qede
\end{remark}

 The relationships amongst our assumptions are summarized in Figure~\ref{fig:relationship}.
 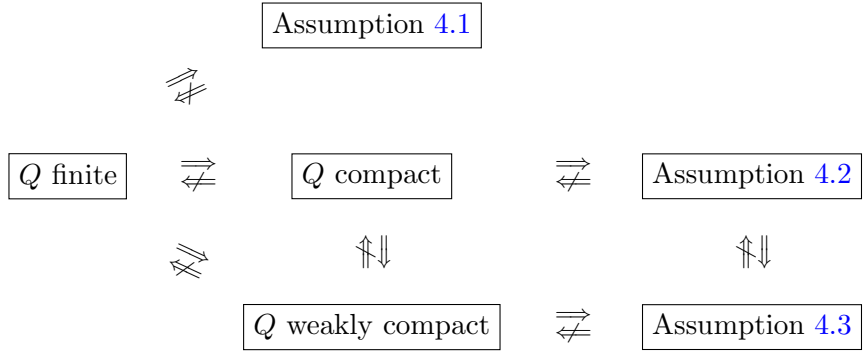
\begin{figure}[!htb]
 \begin{center}
 \begin{tikzpicture}
  \draw (4,2) node[draw,rectangle] (a1) {Assumption~\ref{ass:qwo}};
  \draw (9,0) node[draw,rectangle] (a2) {Assumption~\ref{ass:compactQH}};
  \draw (9,-2) node[draw,rectangle] (a3) {Assumption~\ref{ass:wcompactQH}};
  \draw (4,0) node[draw,rectangle] (compact) {$Q$ compact};
  \draw (4,-2) node[draw,rectangle] (wcompact) {$Q$ weakly compact};
  \draw (0,0) node[draw,rectangle] (finite) {$Q$ finite};

  \draw ($ (a2) !0.47! (compact) $) node {$\substack{\Longrightarrow\\ \centernot\Longleftarrow}$};
  \draw ($ (a3) !0.47! (wcompact) $) node {$\substack{\Longrightarrow\\ \centernot\Longleftarrow}$};
  \draw ($ (compact) !.57! (finite) $) node {$\substack{\Longrightarrow\\ \centernot\Longleftarrow}$};
   \draw ($ (wcompact) !.57! (finite) $) node[below,rotate=-25] {$\substack{\Longrightarrow\\ \centernot\Longleftarrow}$};
  \draw ($ (a1) !.57! (finite) $) node[above,rotate=25] {$\substack{\Longrightarrow\\ \centernot\Longleftarrow}$};
  \draw ($ (wcompact) !0.5! (compact) $) node[rotate=-90] {$\substack{\Longrightarrow\\ \centernot\Longleftarrow}$};
  \draw ($ (a3) !0.5! (a2) $) node[rotate=-90] {$\substack{\Longrightarrow\\ \centernot\Longleftarrow}$};
 \end{tikzpicture}
 \caption{The relationships amongst our assumptions.}\label{fig:relationship}
 \end{center}
 \end{figure}

\section{Examples and Counter-Examples}\label{sec:examples}
 In this section, we give a number of examples which highlight the role of the hyperplane in Theorems~\ref{th:qwo} and~\ref{th:Qcompact}. These examples are interesting in light of results such as \cite{BDMaffine,HLnonconvex} which exploit linear structure to analyze the Douglas--Rachford method.

\begin{example}[Failure for two half-spaces]
 The algorithm no longer remains globally convergent on replacing the half-space by a cone  resulting from the intersection of two half-spaces, as is shown in Figure~\ref{fig:cone}.\qede
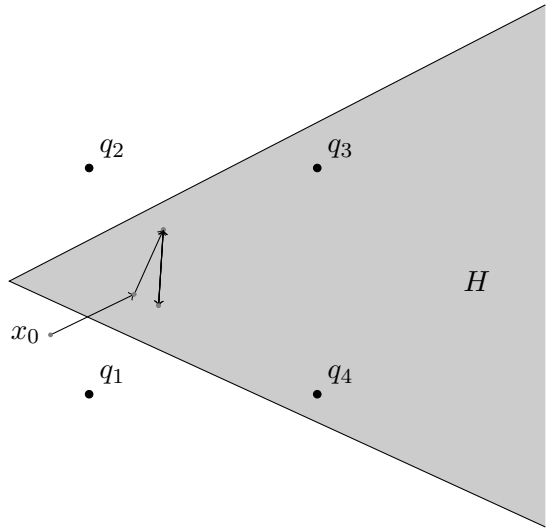
\begin{figure}[ht!]
\begin{center}
	\begin{tikzpicture}[scale=3]
	 \fill[black!20] (2,1.7212) -- (-0.35,0.5) -- (2,-0.5868);
	 \draw[black] (2,1.7212) -- (-0.35,0.5) -- (2,-0.5868);
	 \draw[black] (1.7,0.5) node {$H$};
	 \newcommand{\SetQ}{(0,0),(0,1),(1,1),(1,0)};
	 \foreach \q [count=\j] in \SetQ {
	   \draw \q node[draw,circle,inner sep=1pt,fill=black] {};
	   \draw \q node[above right] (q\j) {$q_{\j}$};
	 };
	  \coordinate (x0) at (-0.1693,0.2624);
	  \draw (x0) node[left] {$x_0$};
	  \newcommand{\DRseq}{(0.1976,0.4418),(0.3256,0.728),(0.3041,0.3924), (0.325,0.7266),(0.3046,0.3916),(0.3247,0.726)};
		 \foreach \x [count=\j] in \DRseq {
		   \pgfmathsetmacro{\jj}{\j-1}
		   \draw \x coordinate (x\j);
		   \draw[black,->] (x\jj) -- (x\j);
		   \draw (x\jj) node[draw=black!50,circle,inner sep=0.5pt,fill=black!50,anchor=center] {};
		 };
	\end{tikzpicture}
    \caption{A $2$-cycle of the Douglas--Rachford algorithm when $H$ is a cone.}\label{fig:cone}
\end{center}
\end{figure}
\end{example}

\begin{example}[Failure for hyperplane]
The Douglas--Rachford operator for $Q$ and $L$ (rather than $H$)  is given by $$\DR(x)=\bigcup_{q\in P_Q(x)}\DR(x,q)=\bigcup_{q\in P_Q(x)}q+(\langle a,x\rangle+b-2\langle a,q\rangle)a.$$
In this case, Proposition~\ref{prop:4} no longer holds, hence the algorithm need not converge. An example with cycling behavior is given in Figure~\ref{fig:hyperplane}.
    \begin{figure}[ht!]
    \begin{center}
    \begin{tikzpicture}[scale=3]
	 \draw[thick,black] (-1.5,0) -- (1.5,0);
	 \draw[black] (1.3,0.2) node {$L$};
	 \newcommand{\SetQ}{(0,1),(1,-1)};
	 \foreach \q [count=\j] in \SetQ {
	   \draw \q node[draw,circle,inner sep=1pt,fill=black] {};
	   \draw \q node[above right] (q\j) {$q_{\j}$};
	 };
	  \coordinate (x0) at (-1,1);
	  \draw (x0) node[left] {$x_0$};
	  \newcommand{\DRseq}{(0,0),(0,-1),(1,0),(1,1),(0,0)};
		 \foreach \x [count=\j] in \DRseq {
		   \pgfmathsetmacro{\jj}{\j-1}
		   \draw \x coordinate (x\j);
		   \draw[black,->] (x\jj) -- (x\j);
		   \draw (x\jj) node[draw=black!50,circle,inner sep=0.5pt,fill=black!50,anchor=center] {};
		 };
	    \draw (x5) node[draw=black!50,circle,inner sep=0.5pt,fill=black!50,anchor=center] {};
	\end{tikzpicture}
	\caption{A $4$-cycle of the Douglas--Rachford algorithm with hyperplane constraint.}\label{fig:hyperplane}
    \end{center}
    \end{figure}
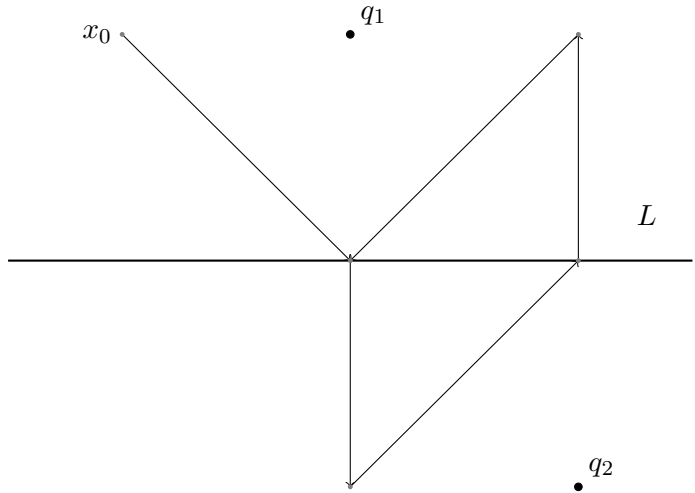
 Moreover, cycles are still possible, even in the  product formulation in terms of the diagonal space, see also Example~\ref{ex:pierra}, as the following example shows. Let $C_1=C_2:=\{0,1\}$. Consider
$$C:=C_1\times C_2=\left\{
\left(\begin{array}{c}
0\\
0
\end{array}
\right),
\left(\begin{array}{c}
0\\
1
\end{array}
\right),
\left(\begin{array}{c}
1\\
0
\end{array}
\right),
\left(\begin{array}{c}
1\\
1
\end{array}
\right)
\right\}$$
and
$$D:=\{(x,y)\in\R^2\mid x=y\}.$$
If $x_0=(-1/2,1)^T$, then $x_1=(1/4,3/4)^T$, $x_2=(3/4,1/4)^T$ and $x_3=x_1$.\qede
\end{example}

\begin{example}[Failure with slab constraints]
  The algorithm no longer remains globally convergent on replacing the half-space by a \emph{slab} constraint. That is,
   $$H := \{x\mid c\geq \langle a,x \rangle  \geq d\},$$
   with $d < c$, both finite.
 An example is given in Figure~\ref{fig:slab}.

     \begin{figure}[ht!]
    \begin{center}
    \begin{tikzpicture}[scale=3]
	 \fill[black!20] (-1.7,-0.06) -- (0.6,-0.06) -- (0.6,-0.59) -- (-1.7,-0.59) -- cycle;
	 \draw[black] (0.4,0.1) node {$H$};
	 \draw[thick,black] (-1.7,-0.06) -- (0.6,-0.06);
	 \draw[thick,black] (0.6,-0.59) -- (-1.7,-0.59) ;
	 \newcommand{\SetQ}{(0.01,-0.35),(-0.3,-0.78),(-0.43,0.01)};
	 \foreach \q [count=\j] in \SetQ {
	   \draw \q node[draw,circle,inner sep=1pt,fill=black] {};
	   \draw \q node[above right] (q\j) {$q_{\j}$};
	 };
	  \coordinate (x0) at (-1,1);
	  \draw (x0) node[left] {$x_0$};
	  \newcommand{\DRseq}{(-0.63,0.28),(-0.43,0.01),(-0.43,-0.06),(-0.43,-0.12),(-0.43,-0.19),(-0.43,-0.26),(-0.43,-0.32),(-0.43,-0.39),(-0.43,-0.45),(-0.3,-0.26),(-0.43,-0.33),(-0.43,-0.39),(-0.43,-0.46),(-0.3,-0.26),(-0.43,-0.33),(-0.43,-0.4),(-0.43,-0.46)};
		 \foreach \x [count=\j] in \DRseq {
		   \pgfmathsetmacro{\jj}{\j-1}
		   \draw \x coordinate (x\j);
		   \draw[black,->] (x\jj) -- (x\j);
		   \draw (x\jj) node[draw=black!50,circle,inner sep=0.5pt,fill=black!50,anchor=center] {};
		 };
	    \draw (x5) node[draw=black!50,circle,inner sep=0.5pt,fill=black!50,anchor=center] {};
	\end{tikzpicture}
	\caption{A $4$-cycle of the Douglas--Rachford algorithm with a slab constraint.}\label{fig:slab}
    \end{center}
    \end{figure}
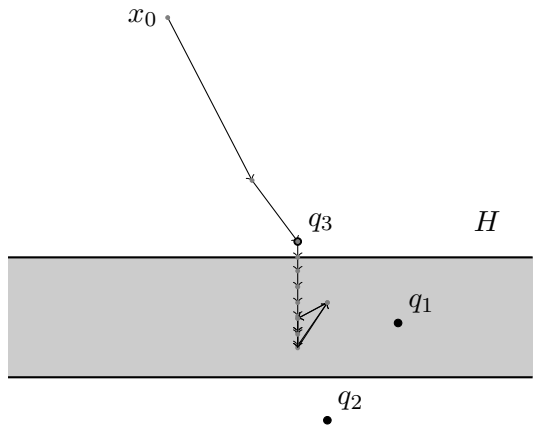
 \qede
\end{example}

\begin{example}[Failure of the product reformulation] \label{ex:pierra}
 \emph{Pierra's product space reformulation} casts any feasibility problem with finitely many constraints as an equivalent two set problem in a larger product space. As the Douglas--Rachford method can only be directly applied to two set feasibility problems, this reformulation is crucial in many applications.

 Consider the constraint sets
  $$H:=\{x\in\R^2\mid x_2\leq 1\},\qquad Q:=\{0,1\}\times\{0,1\}.$$
 Applied to this problem, the global convergence of the Douglas--Rachford method is covered by Theorem~\ref{th:finiteQ}. However, as we will show, the product space reformulation destroys convergence.

 Consider the product space $\R^2\times \R^2$ with the reformulation with constraints
  $$C:=H\times Q=\{(x,y)\mid x\in H,\,y\in Q\},\qquad D=\{(x,y)\mid x=y\}.$$
 Observe that $x\in Q\cap H$ if and only if $(x,x)\in C\cap D$. Furthermore, note that the neither of the sets $C$ or $D$ are half-spaces, hence our results no longer apply.

 Consider the Douglas--Rachford iteration with the reflection first performed with respect to the \emph{diagonal space} $D$ for initial point $(x_0,y_0)=\left((0,2/5),(0,4/5)\right)$. Then $(x_1,y_1)=\left((0,3/5),(0,1/5)\right)$ and $(x_2,y_2)=(x_0,y_0)$. That is, a $2$-cycle is obtained.

 If the reflection was instead performed first with respect to $H\times Q$ for initial point $(x_0,y_0)=\left((0,4/5),(0,2/5)\right)$, a $2$-cycle is still obtained. In this case, we have $(x_1,y_1)=\left((0,1/5),(0,3/5)\right)$ and $(x_2,y_2)=(x_0,y_0)$. \qede
\end{example}

\section{Conclusion}\label{sec:conclusion}

 We have established global convergence and described global behavior of the Douglas--Rachford method applied to the two-set feasibility problem of finding a point in the intersection of a half-space and a second potentially non-convex set, assumed to have  minimal additional structure. The improvement in provable behavior of the method applied to a half-space constraint as compared to affine constraints is quite striking. This improvement is particularly intriguing when one considers that Pierra's product space reformation, which is crucial for many non-convex applications, need only have a constraint with linear structure rather than a half-space.

 The quest for greater understanding of the convergence behavior of the non-convex Douglas--Rachford method would be advanced by finding a unifying  criterion which shows why the method works for sphere and closed hyperplane or closed half-space but not for slabs or cones.

\paragraph{Acknowledgements} F.J. Arag\'on Artacho was supported by MINECO of Spain, as part of the Ram\'on y Cajal program (RYC-2013-13327) and the I+D grant MTM2014-59179-C2-1-P.
J.M. Borwein was supported, in part, by the Australian Research Council.
M.K. Tam was supported by an Australian Post-Graduate Award.

\end{document}